\documentclass[12pt,a4paper]{article}
\usepackage[left=30truemm, right=30truemm]{geometry}
\usepackage{amsmath,amssymb,amsthm}
\usepackage{bm}
\usepackage{mathtools}
\usepackage{xcolor}

\usepackage{enumerate}
\usepackage{enumitem}
\usepackage{float}

\usepackage{tikz}
\usepackage{caption}
\usepackage{subcaption}
\usetikzlibrary{arrows.meta,calc,decorations.markings,math,arrows.meta}

\usepackage[normalem]{ulem}

\definecolor{darkgreen}{RGB}{10,220,10}
\definecolor{rsred}{RGB}{255,100,255}

\newtheorem{theorem}{Theorem}[section]
\newtheorem{proposition}[theorem]{Proposition}
\newtheorem{corollary}[theorem]{Corollary}

\newtheorem{claim}{Claim}[section]

\newtheorem{conjecture}[theorem]{Conjecture}
\newtheorem{observation}[theorem]{Observation}

\usepackage[pdfauthor={derajan},pdfstartview=XYZ,bookmarks=true,
colorlinks=true,linkcolor=blue,urlcolor=magenta,citecolor=blue,pdftex,bookmarks=true,linktocpage=true,hyperindex=true]{hyperref}

\numberwithin{equation}{section}

\usepackage{comment}

\title{Degree-choosability of proper conflict-free list coloring of sparse graphs}
\author{
	Masaki Kashima\thanks{Faculty of Science and Technology, Keio University, Yokohama, Japan. Email: masaki.kashima10@gmail.com} \quad  
	Riste \v{S}krekovski\thanks{Faculty of Mathematics and Physics, University of Ljubljana, Ljubljana, Faculty of Information Studies, Novo Mesto, and Rudolfovo - Science and Technology Centre Novo Mesto, Slovenia. Email:skrekovski@gmail.com} \quad 
	Rongxing Xu\thanks{School of Mathematical and Science, Zhejiang Normal University, Jinhua, China. Email:xurongxing@zjnu.edu.cn}
}

\begin{document}
\maketitle

\begin{abstract}
    Given a graph $G$ and a mapping $f:V(G) \to \mathbb{N}$, an $f$-list assignment of $G$ is a function that maps each $v \in V(G)$ to a set of at least $f(v)$ colors. For an $f$-list assignment $L$ of a graph $G$, a proper conflict-free $L$-coloring of $G$ is a proper coloring $\phi$ of $G$ such that for every vertex $v \in V(G)$, $\phi(v) \in L(v)$ and some appears precisely once in the neighborhood of $v$. We say that $G$ is proper conflict-free $f$-choosable if for every $f$-list assignment $L$ of $G$, there exists a proper conflict-free $L$-coloring of $G$. 
    If $G$ is proper conflict-free $f$-choosable and there is a constant $k$ such that $f(v)= d_G(v)+k$ for every vertex $v$ of $G$, then we say $G$ is proper conflict-free $({\rm degree}+k)$-choosable. 
    In this paper, we consider graphs with a bounded maximum average degree.
    We show that every graph with the maximum average degree less than $\frac{10}{3}$ is proper conflict-free $({\rm degree}+3)$-choosable, and that every graph with the maximum average degree less than $\frac{18}{7}$ is proper conflict-free $({\rm degree}+2)$-choosable.
    As a result, every planar graph with girth at least $5$ is proper conflict-free $({\rm degree}+3)$-choosable, and every planar graph with girth at least $9$ is proper conflict-free $({\rm degree}+2)$-choosable.
\end{abstract}
	
	
\section{Introduction}\label{sec:intro}
Throughout this paper, we only consider simple, finite and undirected graphs. A $k$-coloring of a graph $G$ is a map $\phi: V(G)\to \{1,2,\ldots,k\}$.
A coloring $\phi$ of a graph is called \textit{proper} if every adjacent pair of vertices receive distinct colors.
The chromatic number of a graph $G$, denoted by $\chi(G)$, is the least positive integer $k$ such that $G$ admits a proper $k$-coloring.
	
For a graph $G$, a coloring $\phi$ is called \textit{proper conflict-free} if $\phi$ is a proper coloring of $G$, and for every non-isolated vertex $v\in V(G)$, there exists a color that appears exactly once in the neighborhood of $v$.
The \textit{proper conflict-free chromatic number} of a graph $G$, denoted by $\chi_{\textrm{pcf}}(G)$, is the least positive integer $k$ such that $G$ admits a proper conflict-free $k$-coloring.
The notion of proper conflict-free coloring of a graph was introduced by Fabrici, Lu\v{z}ar, Rindo\v{s}ov\'{a} and Sot\'{a}k~\cite{Fabrici}, where they investigated the proper conflict-free colorings of planar and outerplanar graphs.

For a graph $G$ and a map $f$ from $V(G)$ to positive integers, an \emph{$f$-list assignment} of $G$ is a list $L$ of available colors for each vertices such that $|L(v)|\geq f(v)$ for every $v\in V(G)$.
For a list assignment $L$ of $G$, an $L$-coloring of $G$ is a map $\phi:V(G)\to \bigcup_{v\in V(G)}L(v)$ such that $\phi(v)\in L(v)$ for every $v\in L(v)$.
A graph $G$ is said to be \emph{$f$-choosable} if $G$ has a proper $L$-coloring for any $f$-list assignment $L$ of $G$.
A graph $G$ is said to be \emph{degree-choosable} if $G$ is $f$-choosable with $f(v)=d_{G}(v)$ for each $v \in V(G)$.
It is well known that a connected graph $G$ is not degree-choosable if and only if $G$ is a \emph{Gallai-tree}, i.e. each block of $G$ is either a complete graph or an odd cycle. 
Degree-choosable graphs have been investigated in many papers \cite{Borodin,Borodin1979,Erdos,Thomassen,Vizing}. 

Following this line, we consider the analogous concept of degree-choosability for proper conflict-free coloring.
For a constant $k$, a $({\rm degree}+k)$-list assignment of $G$ is an $f$-list assignment with $f(v)=d_G(v)+k$.
We say that a graph $G$ is \emph{proper conflict-free $({\rm degree}+k)$-choosable} if $G$ is proper conflict-free $L$-colorable for any $({\rm degree}+k)$-list assignment $L$ of $G$.
One fundamental question we asked in \cite{KSX2025-maxdegree} is whether there is an absolute constant $k$ such that every graph is proper conflict-free $({\rm degree}+k)$-choosable.
When $k$ is small, we know that $C_5$ is not proper conflict-free $({\rm degree}+2)$-choosable and that there are infinitely many graphs that are not proper conflict-free $({\rm degree}+1)$-choosable. (See Proposition 4 in \cite{KSX2026}.)
Considering these examples, we pose the following conjectures.

\begin{conjecture}[\cite{KSX2026}]\label{conj:degree plus 2}
    Every connected graph other than the 5-cycle $C_5$ is proper conflict-free $({\rm degree}+2)$-choosable.
\end{conjecture}

\begin{conjecture}\label{conj:degree plus 1}
    Every connected graph of minimum degree at least 3 is proper conflict-free $({\rm degree}+1)$-choosable.
\end{conjecture}

Conjecture~\ref{conj:degree plus 2} is verified for some classes of graphs such as subcubic graphs~\cite{KSX2025-maxdegree}, outerplanar graphs~\cite{KSX2026}, $K_4$-minor free graphs with maximum degree at most 4~\cite{WZ2025}, and outer-1-planar graphs with maximum degree at most 4~\cite{WZ2025}.
Although it seems hard to completely solve Conjectures~\ref{conj:degree plus 2} and \ref{conj:degree plus 1}, the following observation can be shown easily (See \cite{KSX2025-tree} for a proof).
\begin{observation}\label{obs-degeneracy}
	If a graph $G$ is $d$-degenerate for some positive integer $d$, then $G$ is proper conflict-free $({\rm degree} + d+1)$-choosable.
\end{observation}

The above bound is merely an observation, but reducing it by one, as we found, could be quite challenging. We believe that it holds with the single exception of the $5$-cycle $C_5$, as follows.
	
\begin{conjecture}[\cite{KSX2025-tree}]\label{conj:degenericity}
	If a connected graph $G\ne C_5$ is $d$-degenerate for some positive integer $d$, then $G$ is proper conflict-free $({\rm degree}+d)$-choosable.
\end{conjecture}

Conjecture~\ref{conj:degenericity} has been verified for $d=1$ \cite{KSX2025-tree} and some special families of graphs. For example, it is well known that planar graphs are $5$-degenerate and outerplanar graphs are $2$-degenerate. We show that every planar graph is proper conflict-free $({\rm degree}+5)$-choosable \cite{KSX2025-planar} and that every connected outerplanar graph other than $C_5$ is proper conflict-free $({\rm degree}+2)$-choosable \cite{KSX2026}. 

In this paper, we consider graphs with a bounded maximum average degree.
For a graph $G$, the \emph{maximum average degree}, denoted by ${\rm mad}(G)$, is the maximum value of $\frac{2|E(H)|}{|V(H)|}$ where the maximum is taken among all subgraphs $H$ of $G$.
The maximum average degree is known as a parameter to estimate the sparsity of a graph, and hence there are many results on colorings of graphs with a bounded maximum average degree.
Our main results are as follows, and we give proofs of Theorems~\ref{thm:mad 10 over 3} and \ref{thm:mad 18 over 7} in Sections~\ref{sec:mad 10 over 3} and \ref{sec:mad 18 over 7}, respectively.

\begin{theorem}\label{thm:mad 10 over 3}
    Every graph $G$ with ${\rm mad}(G)<\frac{10}{3}$ is proper conflict-free $({\rm degree}+3)$-choosable.
\end{theorem}

\begin{theorem}\label{thm:mad 18 over 7}
    Every connected graph $G$ with ${\rm mad}(G)<\frac{18}{7}$ other than $C_5$ is proper conflict-free $({\rm degree}+2)$-choosable.
\end{theorem}

If a graph $H$ has the average degree at most $\frac{10}{3}<4$, then it is easy to see that $H$ has a vertex of degree at most $3$.
Thus, the assumption of ${\rm mad}(G)<\frac{10}{3}$ implies that $G$ is $3$-degenerate.
Similarly, every graph $G$ with ${\rm mad}(G)<\frac{18}{7}$ is $2$-degenerate.
Thus, Theorems~\ref{thm:mad 10 over 3} and \ref{thm:mad 18 over 7} support Conjecture~\ref{conj:degenericity}.

Furthermore, the following is a direct corollary of our main results and Euler's formula.
	
\begin{corollary}\label{cor:girth}
    Let $G$ be a planar graph.
    \begin{itemize}
        \item If $G$ has girth at least $5$, then $G$ is proper conflict-free $({\rm degree}+3)$-choosable. 
        \item If $G$ has girth at least $9$, then $G$ is proper conflict-free $({\rm degree}+2)$-choosable. 
    \end{itemize}
\end{corollary}

The latter statement improves the girth condition of a recent result in Wang and Zhang~\cite{WZ2025}, where they showed that every planar graph with girth at least 12 is proper conflict-free $({\rm degree}+2)$-choosable.
	  
We conclude this section by introducing the notation. For a subgraph $H$ of a graph $G$ and a vertex $v \in V(G)$ (not necessarily in $H$), we denote $N_H(v)$ the neighbors of $v$ in $V(H)$. For a set $S \subseteq V(G)$, we denote  $N_H(S)$ all the vertices in $V(H)$ which have a neighbor in $S$. 
For a (partial) coloring $\phi$ of $G$ and a subset $R \subseteq V(G)$, let $\phi(R)$ be the set comprised of all colors used on $R$ of $\phi$ (not every vertex in $R$ needs to be colored in $\phi$). For a (partial) proper conflict-free coloring $\phi$ of $G$, let $\mathcal{U}_{\phi}(v,G)$ be the set of colors that appear exactly once at the neighborhood of $v$ in $G$.
For a vertex $u \in G$, we say $u$ is a \emph{$k$-vertex} (\emph{$k^+$-vertex}, \emph{$k^-$-vertex}, respectively) in $G$ if $d_{G}(u)=k$ ($d_{G}(u) \geq k$, $d_{G}(u) \leq k$, respectively).
A $k$-vertex ($k^+$-vertex, $k^-$-vertex, respectively) that is adjacent to a vertex $u$ is called a \emph{$k$-neighbor} (\emph{$k^+$-neighbor}, \emph{$k^-$-neighbor}, respectively) of $u$.

\section{Proof of Theorem \ref{thm:mad 10 over 3}}\label{sec:mad 10 over 3}

Suppose that the statement is false, and let $G$ be a counterexample such that $|V(G)|+|E(G)|$ is as small as possible. Assume that $L$ is a $({\rm degree}+3)$-list assignment of $G$ such that $G$ is not proper conflict-free $L$-colorable. 

\subsection{Forbidden configurations}

We show some claims on forbidden configurations of $G$.
It is easy to see that a 1-vertex with at least three available colors is reducible, and hence we may assume that $\delta(G)\geq 2$.

\begin{claim}\label{claim-triangle}
    A $2$-vertex is not contained in a triangle.
\end{claim}

\begin{proof}
    Suppose that $G$ has a $2$-vertex $u$ that is contained in a triangle $uv_1v_2$.
    Let $G'=G-u$.
    By the minimality of $G$, $G'$ admits a proper conflict-free $L$-coloring $\phi$.
    As $\delta(G)\geq 2$, neither of $\{v_1,v_2\}$ is an isolated vertex of $G'$, so we choose $c_{v_i}\in \mathcal{U}_{\phi}(v_i,G')$ for each $i\in \{1,2\}$.
    We extend $\phi$ to $u$ by choosing $\phi(u)\in L(u)\setminus \{\phi(v_1),\phi(v_2),c_{v_1},c_{v_2}\}$.
    Since $\phi(v_1)\neq \phi(v_2)$, we have that $\phi(v_1)\in \mathcal{U}_{\phi}(u,G)$, and hence $\phi$ is a proper conflict-free $L$-coloring of $G$, a contradiction.
\end{proof}

\begin{claim}\label{claim-4cycle}
    If $u$ and $v$ are both $2$-vertices, then $N_G(u)\neq N_G(v)$.
\end{claim}

\begin{proof}
    Suppose that $G$ has two $2$-vertices $u$ and $v$ such that $N_G(u)=N_G(v)=\{w_1,w_2\}$.
    Let $G'=G-u$.
    By the minimality of $G$, $G'$ admits a proper conflict-free $L$-coloring $\phi$.
    Using a similar argument as in the proof of Claim~\ref{claim-triangle}, we can extend $\phi$ to a proper $L$-coloring of $G$ so that $\mathcal{U}_{\phi}(w,G)\neq \emptyset$ for any $w\in V(G)\setminus \{u\}$.
    Furthermore, since $\mathcal{U}_{\phi}(v,G)\neq \emptyset$, we know that $\phi(w_1)\neq \phi(w_2)$.
    Thus, $\mathcal{U}_{\phi}(u,G)\neq \emptyset$ and $\phi$ is a proper conflict-free $L$-coloring of $G$, a contradiction.
\end{proof}

\begin{claim}\label{claim-33}
    A $3^-$-vertex is not adjacent to another $3^-$-vertex.  
\end{claim}

\begin{proof}  
    Suppose to the contrary that $G$ has two adjacent $3^-$-vertices $u$ and $v$.
    Let $d_G(u)=k+1$ and $d_G(v)=\ell+1$ for some $k,\ell \in \{1, 2\}$. Let $u_1, \ldots, u_k$ be the neighbors of $u$ distinct from $v$, and let $v_1, \ldots, v_{\ell}$ be the neighbors of $v$ distinct from $u$. Let $G' = G - \{u,v\}$. By the minimality of $G$, $G'$ admits a proper conflict-free $L$-coloring $\phi$. 
    By Claim~\ref{claim-triangle}, $u$ and $v$ does not have a common $2$-neighbor, and hence $G'$ does not have an isolated vertex.
    For each $w \in N_{G'}(u) \cup N_{G'}(v)$, we choose a color $c_w \in \mathcal{U}_\phi(w, G')$. Let
    \begin{align*}
        L^*(u) &= L(u) \setminus (\phi(N_{G'}(u)) \cup \{c_w \mid w \in N_{G'}(u)\}) \text{ and}\\
        L^*(v) &= L(v) \setminus (\phi(N_{G'}(v)) \cup \{c_w \mid w \in N_{G'}(v)\}).
    \end{align*}
	Obviously, $|L^*(u)|\geq k+4-2k \geq 2$ and $|L^*(v)|\geq \ell+4-2\ell \geq 2$. We will extend $\phi$ to a proper conflict-free $L$-coloring of $G$.
	
	If $|L^*(u)| = 2$, then it must be that $d_G(u) = 3$ and $\phi(u_1) \neq \phi(u_2)$. Therefore, we can choose $\phi(u) \in L^*(u) \setminus \{\phi(v_1)\}$ and then choose $\phi(v) \in L^*(v) \setminus \{\phi(u)\}$. It is clear that $\phi$ is proper. Since $d_{G}(u) \leq 3$ and $d_{G}(v) \leq 3$,  $\phi(u) \neq \phi(v_1)$ ensures that $\mathcal{U}_{\phi}(v, G) \neq \emptyset$, and $\phi(u_1) \neq \phi(u_2)$ ensures that $\mathcal{U}_{\phi}(u, G) \neq \emptyset$.  Hence, $\phi$ is a proper conflict-free $L$-coloring of $G$, a contradiction.
	
	Thus, we may assume that $|L^*(u)| \geq 3$. Then we choose $\phi(v) \in L^*(v) \setminus \{\phi(u_1)\}$, and subsequently choose $\phi(u) \in L^*(u) \setminus \{\phi(v), \phi(v_1)\}$. By similar arguments as in the last paragraph, $\phi$ is a proper conflict-free $L$-coloring of $G$.
\end{proof}

\begin{claim}\label{claim: number of 2-neighbors}
	For any integer $k\geq 4$, every $k$-vertex of $G$ has at most $k-3$ neighbors of degree $2$. 
\end{claim}
\begin{proof}
	Suppose to the contrary that $u \in V(G)$ with $N_{G}(u)=\{u_1, u_2, \ldots, u_k\}$, where $u_1, u_2, \ldots, u_{k-2}$ have degree $2$. 
    For each $i \in \{1, 2, \ldots, k-2\}$, let $v_i$ be the neighbor of $u_i$ other than $u$. Let $G' = G - \{u, u_1, u_2, \ldots, u_{k-2}\}$. By the minimality of $G$, $G'$ admits a proper conflict-free $L$-coloring $\phi$.
    By Claims~\ref{claim-triangle} and \ref{claim-4cycle}, $u_{k-1}, u_k, v_1, v_2, \ldots v_{k-2}$ are pairwise distinct vertices of $G$, and hence none of them is an isolated vertex of $G'$.
	For each $w \in \{u_{k-1}, u_k, v_1, v_2, \ldots, v_{k-2}\}$, choose a color $c_w \in \mathcal{U}_\phi(w, G')$. We define a list assignment $L^*$ by 
    \begin{itemize}
        \item $L^*(u) = L(u) \setminus \left(\phi(\{u_{k-1}, u_k, v_1, v_2, \ldots, v_{k-2}\}) \cup \{c_{u_{k-1}}, c_{u_k}\}\right)$ and
        \item for each $i\in \{1,\dots , k-2\}$, let $L^*(u_i) = L(u_i) \setminus \{\phi(v_i), c_{v_i}\}$.
    \end{itemize}
	Clearly, we have
	\[
	|L^*(u)| \geq k + 3 - k - 2 = 1,
	\quad\text{and}\quad
	|L^*(u_i)| \geq 5 - 2 = 3.
	\]
	
	Now, we extend $\phi$ to a proper conflict-free $L$-coloring of $G$.
	We first choose $\phi(u) \in L^*(u)$ arbitrarily.  
	If $\phi(u_{k-1}) = \phi(u_k)$, then we choose $\phi(u_1) \in L^*(u_1) \setminus \{\phi(u), \phi(u_k)\}$, and for each $i \in \{2, 3, \ldots, k-2\}$, let
	$\phi(u_i) \in L^*(u_i) \setminus \{\phi(u), \phi(u_1)\}$.
	Otherwise,
    %$\phi(u_{k-1}) \neq \phi(u_k)$, then
    for each $i \in \{1, 2, \ldots, k-2\}$, choose $\phi(u_i) \in L^*(u_i) \setminus \{\phi(u), \phi(u_k)\}$. 
    In either case, it is easy to see that $\phi$ is a proper $L$-coloring such that every vertex $w\in V(G)\setminus \{u\}$ satisfies $\mathcal{U}_{\phi}(w,G)\neq \emptyset$.
    By the choice of $\phi(u_1),\phi(u_2),\dots ,\phi(u_{k-2})$, we have $\phi(u_1)\in \mathcal{U}_{\phi}(u,G)$ when $\phi(u_{k-1})=\phi(u_k)$ and $\phi(u_k)\in \mathcal{U}_{\phi}(u,G)$ when $\phi(u_{k-1})\neq \phi(u_k)$, and hence $\phi$ is a proper conflict-free $L$-coloring of $G$, a contradiction.
\end{proof}

\begin{claim}\label{claim-k-3 small neighbors}
	For any integer $k\geq 4$, if a $k$-vertex $u$ has $k-3$ neighbors of degree $2$, then $u$ has no $3$-neighbors.
\end{claim}
\begin{proof}
    Let $u\in V(G)$ be a $k$-vertex ($k\geq 4$) with $N_{G}(u)=\{u_1, u_2, \ldots, u_k\}$.
	Suppose to the contrary that $d_G(u_1)=d_G(u_2)= \dots =d_G(u_{k-3})=2$ and $d_{G}(u_{k-2})=3$. For each $i \in \{1, 2, \ldots, k - 3\}$, let $v_i$ be the neighbor of $u_i$ other than $u$. Let $G' = G - \{u, u_1, u_2, \ldots, u_{k - 3}\}$. By the minimality of $G$, the graph $G'$ admits a proper conflict-free $L$-coloring $\phi$.
	By Claims~\ref{claim-triangle} and \ref{claim-4cycle}, $u_{k-2}, u_{k-1}, u_k, v_1, v_2, \ldots v_{k-3}$ are pairwise distinct vertices of $G$, and hence none of them is an isolated vertex of $G'$.
    For each $w \in \{u_{k-1}, u_k, v_1, v_2, \ldots v_{k-3}\}$, choose a color $c_w \in \mathcal{U}_\phi(w, G')$. 
    
    We define a list assignment $L^*$ by 
    \begin{itemize}
        \item $L^*(u) = L(u) \setminus \left(\phi(\{u_{k-2}, u_{k-1}, u_k, v_1, v_2, \ldots, v_{k-3}\}) \cup \{c_{u_{k-1}}, c_{u_k}\}\right)$ and
        \item for each $i\in \{1,\dots ,k-3\}$, $L^*(u_i) = L(u_i) \setminus \{\phi(v_i), c_{v_i}\}$.
    \end{itemize}
	Clearly, we have $|L^*(u)| \geq 1$, and $|L^*(u_i)| \geq 3$.
	
	Now, we extend $\phi$ to a proper conflict-free $L$-coloring of $G$. First, choose $\phi(u) \in L^*(u)$ arbitrarily. 
    Note that $\mathcal{U}_{\phi}(u_{k-2},G)\neq \emptyset$ whatever color $\phi(u)$ we assign since $d_G(u_{k-2})=3$ and two different colors appear in $N_G(u_{k-2})\setminus \{u\}$.
    Since $u$ has three colored neighbors $\{u_{k-2},u_{k-1},u_k\}$ at this time, either all of them are colored with a color $\alpha$ or there is a color $\beta$ that appears exactly one of them.

    In the former case, let $\phi(u_1) \in L^*(u_1) \setminus \{\phi(u), \alpha\}$, and for each $i \in \{2, 3, \ldots, k-3\}$, choose $\phi(u_i) \in L^*(u_i) \setminus \{\phi(u), \phi(u_1)\}$. 
    In the latter case, for each $i \in \{1, 2, \ldots, k-3\}$, choose $\phi(u_i) \in L^*(u_i) \setminus \{\phi(u), \beta\}$.
    It is easy to check that $\phi$ is a proper $L$-coloring such that $\mathcal{U}_\phi(w,G)\neq \emptyset$ for every $w\in V(G)\setminus \{u\}$.
    Furthermore, since $\phi(u_1)\in \mathcal{U}_\phi(u,G)$ in the former case and $\beta\in \mathcal{U}_\phi(u,G)$ in the latter case, $\phi$ is a proper conflict-free $L$-coloring of $G$, a contradiction.
\end{proof}

\subsection{Discharging}

For each $v \in V(G)$, define the initial charge as $c(v) = 3d(v) - 10$.

We apply the following discharging rules:  
\begin{itemize} 
	\item  Every $4^+$-vertex sends charge $2$ to each adjacent $2$-vertex.
	\item  Every $4^+$-vertex sends charge $\frac{1}{3}$ to each adjacent $3$-vertex.
\end{itemize}

We shall show that after applying the discharging rules, the final charge $c^*(v) \geq 0$ for all $v \in V(G)$. If $v$ is a $3^-$-vertex, then obviously $c^*(v)=0$ by Claim~\ref{claim-33} and the discharging rules.

Let $v$ be a $k$-vertex with $k \geq 4$. For each positive integer $k$, let $n_k$ be the number of $k$-neighbors of $v$. By Claim~\ref{claim: number of 2-neighbors}, $n_2 \leq k-3$. 

If $n_2 = k-3$, then Claim~\ref{claim-k-3 small neighbors} implies that $n_3=0$, and hence
\[
c^*(v) = c(v) - 2n_2 - \frac{1}{3}n_3 = c(v) - 2(k - 3) = 3k - 10 - 2(k - 3) = k - 4 \geq 0.
\]
Otherwise, since $n_2 \leq k-4$ and $n_2+n_3\leq k$, it follows that 
\[c^*(v)= c(v)-2n_2-\frac{1}{3}n_3\geq c(v)-2(k-4)-\frac{1}{3}\cdot 4=3k-10-2k+8-\frac{4}{3}=k-\frac{10}{3}>0.\]
Therefore, after discharging, every element in $V(G)$ has non-negative charge, and hence 
\[
0\leq \sum_{v\in V(G)}c^*(v)=\sum_{v\in V(G)}c(v)=\sum_{v\in V(G)}(3d_G(v)-10)=6|E(G)|-10|V(G)|.
\]
However, this implies that $\text{mad}(G)\geq \frac{2|E(G)|}{|V(G)|}\geq \frac{10}{3}$, a contradiction.
This completes the proof of Theorem~\ref{thm:mad 10 over 3}.

\section{Proof of Theorem \ref{thm:mad 18 over 7}}\label{sec:mad 18 over 7}

Suppose that the statement is false, and let $G$ be a counterexample such that $|V(G)|+|E(G)|$ is as small as possible. Assume that $L$ is a $({\rm degree}+2)$-list assignment of $G$ such that $G$ is not proper conflict-free $L$-colorable. 
As every cycle of length $\ell\neq 5$ is proper conflict-free $4$-choosable, $G$ is not a cycle.

A \emph{$k$-thread} of $G$ is a path of $G$ with $k$ vertices and each of which has degree $2$ in $G$. Note that if a graph contains a $k$-thread, then it also contains an $\ell$-thread for any $1 \leq \ell \leq k-1$. A $k$-thread is said to be \emph{maximal} if there is no $(k+1)$-thread containing it as a subgraph. For a vertex $u$, we say \emph{$u$ is incident with a $k$-thread} if there is a path $uv_1v_2\cdots v_k$ with $v_1v_2\cdots v_k$ being a $k$-thread. Note that a $1$-thread is just a $2$-vertex.

Before we start to analyze the structure of the minimum counterexample, we show statements on proper conflict-free $({\rm degree}+2)$-choosability of some specified graphs.

\subsection{Auxiliary results}

\begin{proposition}[\cite{KSX2026}]\label{prop:c5}
  Let $C=v_0v_1v_2v_3v_4v_0$ be a $5$-cycle.
  Suppose that $L$ is a $4$-list assignment of $C$.
  Then $C$ is not proper conflict-free $L$-colorable if and only if $L(v_0)=L(v_1)=L(v_2)=L(v_3)=L(v_4)$ and $|L(v_0)|=4$.
\end{proposition}

\begin{proposition}\label{prop:complete bip}
    For a positive integer $r$, the complete bipartite graph $K_{2,r}$ is proper conflict-free $({\rm degree}+2)$-choosable.
\end{proposition}

\begin{proof}
    Let $\{x,y\}$ and $\{u_1,u_2,\dots ,u_r\}$ be the two parts of $K_{2,r}$ and let $L$ be a $({\rm degree}+2)$-list assignment of $K_{2,r}$.
    We first choose three distinct colors $\phi(x)\in L(x)$, $\phi(y)\in L(y)$, and $\phi(u_1)\in L(u_1)$.
    For $i\in \{2,3,\dots ,r\}$, since $|L(u_i)|\geq 4$, we let $\phi(u_i)\in L(u_i)\setminus \{\phi(x),\phi(y),\phi(u_1)\}$.
    By the choice of colors, we have $\phi(u_1)\in \mathcal{U}_{\phi}(x,K_{2,r})$, $\phi(u_1)\in \mathcal{U}_{\phi}(y,K_{2,r})$, and $\phi(x)\in \mathcal{U}_{\phi}(u_i,K_{2,r})$, and thus $\phi$ is a proper conflict-free $L$-coloring of $K_{2,r}$.
\end{proof}

For a tuple $(r,s,t)$ of non-negative integers, let $\Theta_{1, 2\star r, 3\star s, 4\star t}$ be a graph consists of two vertices $x$ and $y$ of degree $r+s+t+1$ and internally disjoint paths $P_0,P_1,\dots ,P_{r+s+t}$ joining $x$ and $y$ such that $|E(P_i)|=1$ for $i=0$, $|E(P_i)|=2$ for $1\leq i\leq r$, $|E(P_i)|=3$ for $r+1\leq i\leq r+s$, and $|E(P_i)|=4$ for $r+s+1\leq i\leq r+s+t$.

\begin{proposition}\label{prop:theta graph}
    Let $(r,s,t)$ be a tuple of non-negative integers such that $r+s\geq 1$ and $r+s+t\geq 2$.
    Then $\Theta_{1,2\star r, 3\star s, 4\star t}$ is proper conflict-free $({\rm degree}+2)$-choosable.
\end{proposition}

\begin{proof}
    We fix a tuple $(r,s,t)$ of non-negative integers with $r+s\geq 1$ and $r+s+t\geq 2$.
    Let $x$ and $y$ be two adjacent vertices of degree $r+s+t+1$ and let $P_1,\dots ,P_{r+s+t}$ be internally disjoint paths joining $x$ and $y$ such that 
    \begin{equation*}
        P_i=\begin{cases}
            xu_iy\quad &\text{if}\quad 1\leq i\leq r,\\
            xu_iv_iy\quad &\text{if}\quad r+1\leq i\leq r+s, \quad\text{and}\\
            xu_iv_iw_iy\quad &\text{if}\quad r+s+1\leq i\leq r+s+t.
        \end{cases}
    \end{equation*}
    Let $L$ be a $({\rm degree}+2)$-list assignment of $\Theta_{1, 2\star r, 3\star s, 4\star t}$.
    By deleting extra colors, we may assume that $|L(x)|=|L(y)|=r+s+t+3$ and $|L(u_i)|=|L(v_i)|=|L(w_i)|=4$.

    Let $\mathcal{C}:=\{\{\alpha,\beta\}\mid \alpha\neq \beta, \alpha\in L(x), \beta\in L(y)\}$.
    Since $|L(x)|=|L(y)|=r+s+t+3$, we have that $|\mathcal{C}|\geq \binom{r+s+t+3}{2}$.

    For each set of two colors $\{\alpha,\beta\}\in \mathcal{C}$, let 
    \[I_{\{\alpha,\beta\}}:=\left\{i\mid r+s+1\leq i\leq r+s+t, \{\alpha,\beta\}\subseteq L(v_i)\right\}.\]
    
    Since $|L(v_i)|=4$ for each $i\in \{r+s+1,\dots ,r+s+t\}$, there are at most $\binom{4}{2}=6$ pairs of distinct colors in $\mathcal{C}$ that are contained in $L(v_i)$.
    By the assumptions $r+s\geq 1$ and $r+s+t\geq 2$, we have
    \begin{align*}
        \sum_{\{\alpha,\beta\}\in \mathcal{C}} |I_{\{\alpha,\beta\}}|\leq \sum_{i=r+s+1}^{r+s+t}\left|\binom{L(v_i)}{2}\right|=6t<\frac{(r+s+t+3)(r+s+t+2)}{2}=\binom{r+s+t+3}{2}\leq |\mathcal{C}|,
    \end{align*}
    which implies that there is a pair $\{\alpha,\beta\}\in \mathcal{C}$ such that $I_{\{\alpha,\beta\}}=\emptyset$.
    Without loss of generality, we may assume that $\alpha\in L(x)$ and $\beta\in L(y)$.
    Let $\phi(x)=\alpha$ and $\phi(y)=\beta$.
    For other vertices, we choose the colors as follows.
    \begin{itemize}
        \item For each $i\in \{1,\dots ,r\}$, let $\phi(u_i)\in L(u_i)\setminus \{\alpha,\beta\}$.
        \item For each $i\in \{r+1,\dots ,r+s\}$, let $\phi(u_i)\in L(u_i)\setminus \{\alpha,\beta\}$ and $\phi(v_i)\in L(v_i)\setminus \{\alpha,\beta,\phi(u_i)\}$.
        \item For each $i\in \{r+s+1,\dots ,r+s+t\}$, let $\phi(u_i)\in L(u_i)\setminus \{\alpha,\beta\}$, $\phi(w_i)\in L(w_i)\setminus \{\alpha,\beta,\phi(u_i)\}$, and $\phi(v_i)\in L(v_i)\setminus \{\alpha,\beta,\phi(u_i),\phi(w_i)\}$, where $\phi(v_i)$ can be chosen since $|L(v_i)\cap \{\alpha,\beta\}|\leq 1$.
    \end{itemize}
    Then we have $\beta\in \mathcal{U}_{\phi}(x,\Theta_{1,2\star r, 3\star s, 4\star t})$ and $\alpha\in \mathcal{U}_{\phi}(y,\Theta_{1,2\star r, 3\star s, 4\star t})$, and it is easy to check that $\phi$ is a proper conflict-free $L$-coloring.
\end{proof}
 
\subsection{Forbidden configurations}

In the following claims, we give some forbidden configurations of the minimum counterexample $G$. 
It is easy to see that the vertex of degree 1 is reducible, and thus we may assume that $\delta(G)\geq 2$.

\begin{claim}\label{claim:3vertex 2thread}
	A 3-vertex is not incident with a 2-thread of $G$.
\end{claim} 
\begin{proof}
	Suppose that a 3-vertex $u$ is incident with a 2-thread $v_1v_2$ with $uv_1\in E(G)$.
    Let $v$ be the neighbor of $v_2$ distinct from $v_1$.
    Note that $|L(v_i)|\geq 4$ for $i\in \{1,2\}$.
    Let $G'=G-\{v_1,v_2\}$.
    Then $G'$ is proper conflict-free $L$-colorable; if $G'$ has a component that is isomorphic to $C_5$, then it follows from Proposition~\ref{prop:c5} and the fact that every component of $G'$ has a vertex $x$ with $|L(x)|\geq d_G(x)+2\geq d_{G'}(x)+3$, otherwise, it follows from the minimality of $G$.
    Let $\phi$ be a proper conflict-free $L$-coloring of $G'$.
    As $d_{G'}(v)\geq\delta(G)-1\geq 1$, we choose $c_{v} \in \mathcal{U}_{\phi}(v,G')$.
    We extend $\phi$ to an $L$-coloring of $G$ by choosing $\phi(v_2)\in L(v_2)\setminus \{\phi(u),\phi(v),c_v\}$ and choosing $\phi(v_1)\in L(v_1)\setminus \{\phi(v_2),\phi(u),\phi(v)\}$.
    Since $d_G(u)=3$ and at least two colors appear in the neighbors of $u$, $\phi$ is a proper conflict-free $L$-coloring of $G$, a contradiction.
\end{proof}

\begin{claim}\label{claim:long thread}
	$G$ has no $4$-thread.
\end{claim}
\begin{proof}
	Suppose that $P = v_1v_2v_3v_4$ is a $4$-thread, where $u$ is the neighbor of $v_1$ distinct from $v_2$ and $v$ is the neighbor of $v_4$ distinct from $v_3$. Let $G' = G - V(P)$. By Proposition~\ref{prop:c5} and the minimality of $G$, the graph $G'$ admits a proper conflict-free $L$-coloring $\phi$. For each $w \in \{u, v\}$, choose a color $c_w \in \mathcal{U}_\phi(w, G')$. Let $L^*(v_1) = L(v_1) \setminus  \{c_u, \phi(u)\}$, $L^*(v_2) =  L(v_2) \setminus  \{\phi(u)\}$, $L^*(v_3) =  L(v_3) \setminus  \{\phi(v)\}$ and $L^*(v_4) = L(v_4) \setminus  \{c_v, \phi(v)\}$. Note that $|L^*(v_i)| \geq 2$ for $i\in \{1,4\}$ and $|L^*(v_j)| \geq 3$ for $j \in \{2,3\}$. We may assume that $|L^*(v_1)| =2$ by deleting extra colors from the list. We then extend $\phi$ to a coloring of $G$ as follows: First, choose $\phi(v_2) \in L^*(v_2) \setminus L^*(v_1)$;  
	then choose $\phi(v_4) \in L^*(v_4) \setminus \{\phi(v_2)\}$;  
	next choose $\phi(v_3) \in L^*(v_3) \setminus \{\phi(v_2), \phi(v_4)\}$;  
	and finally, choose $\phi(v_1) \in L^*(v_1) \setminus \{\phi(v_3)\}$. It is clear that $\phi$ is proper by the choices. We have that $c_w \in \mathcal{U}_{\phi}(w, G)$ for each of $w \in \{u,v\}$, and the two neighbors of each $v_i$ received two distinct colors in $\phi$. Therefore, $\phi$ is a proper conflict-free coloring of $G$, a contradiction. 
\end{proof}

\begin{claim}\label{claim:pendant cycle}
    $G$ has no cycle $u_1u_2\cdots u_{\ell}u_1$ with $d_G(u_2)=\cdots =d_G(u_{\ell})=2$.
\end{claim}

\begin{proof}
    Suppose that $G$ has a cycle $u_1u_2\cdots u_{\ell}u_1$ with $d_G(u_2)=\cdots =d_G(u_{\ell})=2$.
    Since $G$ has no $4$-thread by Claim~\ref{claim:long thread}, we know that $\ell\in \{3,4\}$.

    Let $G'=G-\{u_2, u_3, \dots ,u_{\ell}\}$.
    By Proposition~\ref{prop:c5} and the minimality of $G$, $G'$ admits a proper conflict-free $L$-coloring $\phi$.
    Since $G$ is not a cycle, we have $d_{G'}(u_1)=d_G(u_1)-2\geq 1$, and hence let $c_{u_1}$ be a color in $\mathcal{U}_{\phi}(u_1,G')$.

    If $\ell=3$, then we let $\phi(u_2)\in L(u_2)\setminus \{\phi(u_1),c_{u_1}\}$ and $\phi(u_3)\in L(u_3)\setminus \{\phi(u_1),\phi(u_2),c_{u_1}\}$.
    If $\ell=4$, then we let $\phi(u_2)\in L(u_2)\setminus \{\phi(u_1),c_{u_1}\}$, $\phi(u_4)\in L(u_4)\setminus \{\phi(u_1),\phi(u_2),c_{u_1}\}$, and $\phi(u_3)\in L(u_3)\setminus \{\phi(u_1),\phi(u_2),\phi(u_3)\}$.
    For each case, it is easy to verify that the coloring $\phi$ is a proper conflict-free $L$-coloring of $G$, a contradiction.
\end{proof}

\begin{claim}\label{claim:3plus neighbor}
	Every $3^+$-vertex is adjacent to at least one $3^+$-vertex.
\end{claim}

\begin{proof}
	Suppose that there is a $k$-vertex ($k\geq 3$) $u_0 \in V(G)$ with $N_{G}(u_0) = \{u_1,\ldots, u_k\}$ such that $d_{G}(u_i) =2$ for each $i\in \{1,\dots ,k\}$. 
    We fix such a vertex $u_0$ so that the degree $k$ is as small as possible.
    By deleting extra colors, we may assume that $|L(u_i)|=4$ for each $i$. For each color $c\in L(u_0)$, let $U_c=\{u_i \mid  c \in L(u_i),\; 1 \leq i\leq k\}$. We first claim that there is a color $\alpha\in L(u_0)$ such that $|U_{\alpha}|\leq \min\{3,k-1\}$. Indeed, if $k=3$, then we may choose a color $\alpha \in L(u_0) \setminus L(u_1)$, which satisfies that $U_{\alpha} \subseteq \{u_2, u_3\}$. If $k \geq 4$, then since 
    $\sum_{c\in L(u_0)}|U_c|\leq \sum_{i=1}^{k} |L(u_i)|=4k$, there is a color $\alpha$ such that $|U_{\alpha}|\leq \left\lfloor\frac{4k}{|L(u_0)|}\right\rfloor\leq \left\lfloor\frac{4k}{k+2}\right\rfloor\leq 3$, as desired.
    Without loss of generality, we may assume that $u_1 \notin U_{\alpha}$.
	
	Let $v_i$ be the neighbor of $u_i$ distinct from $u_0$. By Claim~\ref{claim:pendant cycle}, we know that $v_i\neq u_j$ for any distinct $i$ and $j$. Note that it is possible that $v_i=v_j$ for distinct $i$ and $j$.
    Let $G'=G-\{u_0,u_1,\dots ,u_k\}$.
    If $G'$ has an isolate vertex $v$, then it must be that $N_G(v)\subseteq \{u_1,u_2,\dots ,u_k\}$.
    By the minimality of $k$, we have that $N_G(v)=\{u_1,u_2,\dots ,u_k\}$.
    This implies that $G$ is a complete bipartite graph with parts $\{u_0,v\}$ and $\{u_1,u_2,\dots ,u_k\}$, and $G$ is proper conflict-free $L$-colorable by Proposition~\ref{prop:complete bip}, a contradiction.

    Thus, $G'$ has no isolated vertices.
    Let $L'$ be a list assignment of $G'$ such that $L'(v_i)=L(v_i)\setminus\{\alpha\}$ for each $i\in \{1,\dots, k\}$ and $L'(x)=L(x)$ for every vertex $x \in V(G')\setminus\{v_1,\ldots, v_k\}$. 
    Then we have $|L'(v)|\geq d_{G'}(v)+2$ for every vertex $v\in V(G')$.
    In particular, if a vertex $v\in V(G')$ is adjacent to at least two vertices of $\{u_1,u_2,\dots ,u_k\}$, then we have $|L'(v)|\geq |L(v)|-1\geq d_G(v)+2-1\geq d_{G'}(v)+3$.
    We say a component $H$ of $G'$ is \emph{bad} if $H$ is isomorphic to $C_5$ and $|L'(v)|=4$ for every vertex $v\in V(H)$.
    Note that every vertex of a bad component $H$ is adjacent to at most one vertex in $\{u_1,u_2,\dots ,u_k\}$, and hence each vertex of $H$ has degree 2 or 3 in $G$.
    
    Now we define a proper $L'$-coloring $\phi$ and a color $c_i$ for each $i\in \{1,2,\dots ,k\}$ as follows.
    \begin{enumerate}[label=(\alph*)]
        \item If $G'$ is proper conflict-free $L'$-colorable, then let $\phi$ be a proper conflict-free $L'$-coloring of $G'$ and let $c_i$ be a color in $\mathcal{U}_{\phi}(v_i,G')$ for each $i\in \{1,2,\dots ,k\}$.
        \item Otherwise, $G'$ has a bad component.
        For each bad component $H$ of $G'$, we let $\phi$ be a proper $L'$-coloring of $H$ such that at most one vertex $v_H$ satisfies $\mathcal{U}_{\phi}(v_H,G')=\emptyset$.
        For other components, let $\phi$ be a proper conflict-free $L'$-coloring.

        For each $i\in \{1,2,\dots ,k\}$, if $\mathcal{U}_{\phi}(v_i,G')\neq \emptyset$, then let $c_i$ be a color in $\mathcal{U}_{\phi}(v_i,G')$.
        If $\mathcal{U}_{\phi}(v_i,G')=\emptyset$, then $v_i=v_H$ for some bad component $H$ and $d_{G'}(v_H)=2$. Thus, let $c_i$ be the unique color that appears in $N_{G'}(v_i)$ by $\phi$.
    \end{enumerate}

	We extend $\phi$ to a proper conflict-free $L$-coloring of $G$. Let $\phi(u_0)=\alpha$, and we choose $\phi(u_i)\in L(u_i)\setminus\{\phi(v_i), c_i, \alpha\}$ arbitrarily for each $u_i \in U_{\alpha}$. We define a color $\beta$ as follows: If $U_{\alpha}=\emptyset$, then let $\beta$ be a color in $L(u_1)\setminus\{\phi(v_1), c_1, \alpha\}$ and set $\phi(u_1)=\beta$. If $U_{\alpha}\neq\emptyset$ and $|\phi(U_{\alpha})|=1$, then since $\alpha\notin L(u_1)$, let $\beta$ be a color in $L(u_1)\setminus(\{\phi(v_1), c_1, \alpha\}\cup \phi(U_{\alpha}))$. Otherwise, i.e. $U_{\alpha}\neq\emptyset$ and $|\phi(U_{\alpha})|\geq 2$, the fact $|U_{\alpha}|\leq 3$ implies that some color appears exactly once at $U_{\alpha}$, so let $\beta$ be the color. 
    Finally, for each uncolored vertex $u_j$, since $u_j \notin U_{\alpha}$, we can choose a color $\phi(u_j)\in L(u_j)\setminus\{\phi(v_j), c_j, \alpha, \beta\}$. 
    
    It is easy to verify that $\phi$ is a proper $L$-coloring of $G$, and that $\mathcal{U}_{\phi}(w,G)=\mathcal{U}_{\phi}(w, G') \neq \emptyset$ for each $w\in V(G')\setminus\{v_1, \ldots, v_k\}$. 
    Since $\phi(u_i)\neq c_i$, we have $c_i\in \mathcal{U}_{\phi}(v_i, G)$ when $c_i\in \mathcal{U}_{\phi}(v_i,G')$, and we have $\phi(u_i)\in \mathcal{U}_{\phi}(v_i,G)$ when $\mathcal{U}_{\phi}(v_i,G')=\emptyset$.
    By the definition of $L'$, we have $\phi(v_i) \neq \alpha = \phi(u_0)$, so $\alpha \in \mathcal{U}_{\phi}(u_i, G)$ for each $i \in \{1,\ldots, k\}$. Furthermore, by the choice of colors $\{\phi(u_i)\mid 1\leq i\leq k\}$, we have $\beta\in \mathcal{U}_{\phi}(u_0, G)$. Hence $\phi$ is a proper conflict-free $L$-coloring of $G$, a contradiction.
\end{proof}

\begin{claim}\label{claim:3thread 1}
	For any $k\geq 4$, every $k$-vertex is incident with at most $k-3$ 3-threads.
\end{claim}

\begin{proof}
	Let $u_0$ be a $k$-vertex with neighbors $u_1,\dots ,u_k$ for some $k\geq 4$.
	Suppose that each of $u_1, \ldots, u_{k-2}$ is contained in a 3-thread of $G$.
    We fix such a vertex $u_0$ so that the degree $k$ is as small as possible.
	For each $i\in\{1,\dots ,k-2\}$, let $u_iv_iw_i$ be a 3-thread containing $u_i$ and let $x_i$ be the neighbor of $w_i$ distinct from $v_i$.
    Note that $\{u_i,v_i,w_i,x_i\}\cap \{u_j,v_j,w_j\}=\emptyset$ for any distinct $i$ and $j$ by Claim~\ref{claim:pendant cycle}.
	Let $G'=G-\{u_0\}\cup \bigcup_{i=1}^{k-2}\{u_i,v_i,w_i\}$.
    If $G'$ has an isolated vertex $v$, then it must be that $N_G(v)\subseteq \{u_0,w_1,w_2,\dots ,w_{k-2}\}$, and hence $v$ is a vertex of degree at most $k-1$ that is incident with at least $d_G(v)-1$ 3-threads, a contradiction by the minimality of $k$.

    Thus, $G'$ has no isolated vertices. 
    By Proposition~\ref{prop:c5} and the minimality of $G$, $G'$ admits a proper conflict-free $L$-coloring $\phi$.
	For each $v\in \{u_{k-1}, u_k, x_1,\dots ,x_{k-2}\}$, let $c_v$ be a color in $\mathcal{U}_\phi(v,G')$.
	We define a list assignment $L^*$ of remaining vertices as follows:
	\begin{itemize}
		\item Let $L^*(u_0)=L(u_0)\setminus \{\phi(u_{k-1}), \phi(u_k), c_{u_{k-1}}, c_{u_k}\}$.
		\item For each $i\in \{1,\dots ,k-2\}$, let $L^*(u_i)=L(u_i)$, $L^*(v_i)=L(v_i)\setminus\{\phi(x_i)\}$, and $L^*(w_i)=L(w_i)\setminus \{\phi(x_i),c_{x_i}\}$.
	\end{itemize}
	By deleting extra colors, we may assume that $|L^*(u_i)|=4$, $|L^*(v_i)|=3$, and $|L^*(w_i)|=2$ for each $i\in \{1,\dots ,k-2\}$.
	For each $i\in \{1,\dots ,k-2\}$, let $C(u_i)$ be a set of two colors in $L^*(u_i)\setminus L^*(w_i)$.
	We shall show that we can extend $\phi$ to $\{u_0,u_1,\dots ,u_{k-2}\}$ properly so that $\mathcal{U}_{\phi}(u_0,G)\neq \emptyset$ and $\phi(u_i)\in C(u_i)$ for each $i\in \{1,\dots ,k-2\}$.
	
	\medskip
	\noindent
	\textit{Case 1.} $\phi(u_{k-1})=\phi(u_k)$.
	\medskip
	
	Let $\alpha=\phi(u_{k-1})=\phi(u_k)$.
	We first choose $\phi(u_1)\in C(u_1)\setminus \{\alpha\}$, and for each $i\in \{2,\dots ,k-2\}$, we choose $\phi(u_i)\in C(u_i)\setminus\{\phi(u_1)\}$. 
	By the assumption of this case, we have $|L^*(u_0)|\geq k+2-3=k-1$, and hence we can choose $\phi(u_0)\in L^*(u_0)\setminus \{\phi(u_1),\dots ,\phi(u_{k-2})\}$, which results in a desired coloring of $\{u_0,u_1,\dots ,u_{k-2}\}$.
	
	\medskip
	\noindent
	\textit{Case 2.} $\phi(u_{k-1})\neq \phi(u_k)$.
	\medskip
	
	We fix a color $\alpha\in L^*(u_0)$, and let $U_\alpha:=\{u_i\mid 1\leq i\leq k-2,\; \alpha\in C(u_i)\}$.
	
	If $|U_\alpha|\geq 2$, then we let $\phi(u_i)=\alpha$ for every $u_i\in U_\alpha$, and choose $\phi(u_i)\in C(u_i)\setminus \{\phi(u_k)\}$ for each $u_i\in \{u_1,\dots ,u_{k-2}\}\setminus U_\alpha$.
	Since $|L^*(u_0)|\geq k-2$ and $|\{\phi(u_1),\dots ,\phi(u_{k-2})\}|\leq k-2-(|U_\alpha|-1)\leq k-3$, we can choose $\phi(u_0)\in L^*(u_0)\setminus \{\phi(u_1),\dots ,\phi(u_{k-2})\}$.
	As $\phi(u_{k-1})\neq \phi(u_k)$ and $\alpha\neq \phi(u_k)$, we have $\phi(u_k)\in\mathcal{U}_\phi(u_0,G)$, and hence this is a desired coloring.
	
	Thus, we assume that $|U_\alpha|\leq 1$.
	Then we let $\phi(u_0)=\alpha$ and choose $\phi(u_i)\in C(u_i)\setminus \{\alpha\}$ for each $u_i\in U_\alpha$ if $U_\alpha\neq \emptyset$.
	Since $\phi(u_{k-1})\neq \phi(u_k)$ and $|U_\alpha|\leq 1$, at least one of $\phi(u_{k-1})$ and $\phi(u_k)$, say $\phi(u_k)$, appears exactly once around $u_0$.
	For each $u_i\in \{u_1,\dots ,u_{k-2}\}\setminus U_\alpha$, we choose $\phi(u_i)\in C(u_i)\setminus \{\phi(u_k),\alpha\}$.
	Since $\phi(u_k)\in \mathcal{U}_\phi(u_0,G)$, this is a desired coloring.
	
	\medskip
	Now we color remaining vertices $\bigcup_{i=1}^{k-2}\{v_i,w_i\}$.
	For each $i\in \{1,\dots ,k-2\}$, we first choose $\phi(v_i)\in L^*(v_i)\setminus \{\phi(u_0),\phi(u_i)\}$ and then choose $\phi(w_i)\in L^*(w_i)\setminus \{\phi(v_i)\}$.
	Since $\phi(u_i)\notin L^*(w_i)$, we have $\mathcal{U}_\phi(v_i,G)\neq \emptyset$, and hence $\phi$ is a proper conflict-free $L$-coloring of $G$, a contradiction.
\end{proof}

\begin{claim}\label{claim:3thread 2}
	For any $k\geq 4$, if a $k$-vertex is adjacent to exactly $k-1$ 2-neighbors, then at most $k-4$ of them are contained in 3-threads of $G$.
\end{claim}

\begin{proof}
	Let $u_0$ be a $k$-vertex with neighbors $u_1,\dots ,u_k$ for some $k\geq 4$.
	Suppose that $u_1,\dots , u_{k-1}$ are 2-vertices and that each of $u_1, \dots , u_{k-3}$ is contained in a 3-thread of $G$.
    We take such a vertex $u_0$ so that the degree $k$ is as small as possible.
	For each $i\in\{1,\dots ,k-3\}$, let $u_iv_iw_i$ be a 3-thread containing $w_i$, and let $x_i$ be the neighbor of $w_i$ distinct from $v_i$.
	Let $v_i$ be the neighbor of $u_i$ other than $u_0$ for $i\in \{k-2,k-1\}$.
    By Claim~\ref{claim:pendant cycle}, $\{u_i,v_i,w_i,x_i\}\cap \{u_j,v_j,w_j\}=\emptyset$ for any distinct $i$ and $j$.
	Let $G'=G-\{u_0,u_{k-2},u_{k-1}\}\cup \bigcup_{i=1}^{k-3}\{u_i,v_i,w_i\}$.
    If $G'$ has an isolated vertex $v$, then it must be that $N_G(v)\subseteq \{u_k, v_{k-2}, v_{k-1}\}\cup \{w_1,\dots ,w_{k-3}\}$.
    By Claims~\ref{claim:3plus neighbor}, \ref{claim:3thread 1} and the minimality of $k$, we conclude that $N_G(v)=\{u_k, v_{k-2}, v_{k-1}\}\cup \{w_1,\dots ,w_{k-3}\}$.
    This implies that $G$ is isomorphic to $\Theta_{1,2\star 2, 3\star 0, 4\star (k-3)}$, and hence $G$ is proper conflict-free $L$-colorable by Proposition~\ref{prop:theta graph}, a contradiction.
    
    Thus $G'$ has no isolated vertices.
    By Proposition~\ref{prop:c5} and the minimality of $G$, $G'$ admits a proper conflict-free $L$-coloring $\phi$.
	For each $v\in \{u_k,v_{k-2},v_{k-1}\}\cup \{x_1,\dots ,x_{k-3}\}$, let $c_v$ be a color in $\mathcal{U}_\phi(v,G')$.
	We define a list assignment $L^*$ of remaining vertices as follows:
	\begin{itemize}
		\item Let $L^*(u_0)=L(u_0)\setminus \{\phi(u_k), \phi(v_{k-2}), \phi(v_{k-1}), c_{u_k}\}$.
		\item For each $i\in \{1,\dots ,k-3\}$, let $L^*(u_i)=L(u_i)$, $L^*(v_i)=L(v_i)\setminus\{\phi(x_i)\}$, and $L^*(w_i)=L(w_i)\setminus \{\phi(x_i),c_{x_i}\}$.
		\item For each $i\in \{k-2,k-1\}$, let $L^*(u_i)=L(u_i)\setminus \{\phi(v_i),c_{v_i}\}$.
	\end{itemize}
	By deleting extra colors, we may assume that $|L^*(u_i)|=4$, $|L^*(v_i)|=3$, $|L^*(w_i)|=2$ for each $i\in \{1,\dots ,k-3\}$, and $|L^*(u_{k-2})|=|L^*(u_{k-1})|=2$.
	For each $i\in \{1,\dots ,k-1\}$, let $C(u_i)$ be a set of two colors in $L^*(u_i)\setminus L^*(w_i)$ if $i\leq k-3$ and let $C(u_i)=L^*(u_i)$ if $i\in \{k-2,k-1\}$.
	As we did in the proof of Claim~\ref{claim:3thread 1}, we shall show that we can extend $\phi$ to $\{u_0,u_1,\dots ,u_{k-1}\}$ properly so that $\mathcal{U}_{\phi}(u_0,G)\neq \emptyset$ and $\phi(u_i)\in C(u_i)$ for each $i\in \{1,\dots ,k-1\}$.
	
	We fix a color $\alpha\in L^*(u_0)$ and let $U_\alpha:=\{u_i\mid 1\leq i\leq k-1,\; \alpha\in C(u_i)\}$.
    Note that $\alpha\neq \phi(u_k)$ by the definition of $L^*(u_0)$.
	
	If $|U_\alpha|\geq 3$, then we let $\phi(u_i)=\alpha$ for every $u_i\in U_\alpha$, and choose $\phi(u_i)\in C(u_i)\setminus \{\phi(u_k)\}$ for each $u_i\in \{u_1,\dots ,u_{k-1}\}\setminus U_\alpha$.
	Since $|L^*(u_0)|\geq k-2$ and $|\{\phi(u_1),\dots ,\phi(u_{k-1})\}|\leq k-1-(|U_\alpha|-1)\leq k-3$, we can choose $\phi(u_0)\in L^*(u_0)\setminus \{\phi(u_1),\dots ,\phi(u_{k-1})\}$.
	Obviously we have $\phi(u_k)\in\mathcal{U}_\phi(u_0,G)$, and hence this is a desired coloring.
	
	Thus, we assume that $|U_\alpha|\leq 2$.
	Then we let $\phi(u_0)=\alpha$ and choose $\phi(u_i)\in C(u_i)\setminus \{\alpha\}$ for each $u_i\in U_\alpha$ if $U_\alpha\neq \emptyset$.
	Since $|\{u_k\}\cup U_\alpha|\leq 3$, either all vertices of $\{u_k\}\cup U_\alpha$ is colored in the same color $\beta$, or there exists a color $\gamma$ that appears exactly once in $\{u_k\}\cup U_\alpha$.
	
	In the former case, we choose $\phi(u_{i_0})\in C(u_{i_0})\setminus \{\beta\}$ for some $u_{i_0}\notin U_\alpha$, and let $\phi(u_i)\in C(u_i)\setminus \{\phi(u_{i_0})\}$ for each $u_i\in \{u_1,\dots ,u_{k-1}\}\setminus (U_\alpha\cup \{u_{i_0}\})$. 
	Since $\phi(u_{i_0})\in \mathcal{U}_\phi(u_0,G)$, this is a desired coloring.
	In the latter case, for each $u_i\in \{u_1,\dots ,u_{k-1}\}$, we choose $\phi(u_i)\in C(u_i)\setminus\{\gamma\}$.
	Then, we have $\gamma\in \mathcal{U}_\phi(u_0,G)$, and hence this is a desired coloring again.
	
	By a similar argument as in the proof of Claim~\ref{claim:3thread 2}, we can color the remaining vertices $\bigcup_{i=1}^{k-3}\{v_i,w_i\}$ to obtain a proper conflict-free $L$-coloring of $G$, a contradiction.
\end{proof}

\begin{claim}\label{claim:3vertex bad neighbor}
	Let $u_0$ be a 3-vertex with $N_G(u_0)=\{u_1,u_2,u_3\}$.
	If both $u_1$ and $u_2$ are 2-vertices, then each of them has a $4^+$-neighbor.
\end{claim}

\begin{proof}
	For $i=1,2$, let $v_i$ be the neighbor of $u_i$ distinct from $u_0$.
    By Claim~\ref{claim:3vertex 2thread}, $d_G(v_i)\geq 3$ for $i=1,2$.
    Suppose to the contrary that $d_G(v_1)=3$.
	We may assume that $|L(u_0)|=5$ and $|L(u_1)|=|L(u_2)|=4$.
	Let $G'=G-\{u_0,u_1,u_2\}$.
    If $G'$ has an isolated vertex $v$, then it must be that $v=v_1=v_2=u_3$ and $V(G)=\{u_0,u_1,u_2,u_3\}$.
    These imply that $G$ is isomorphic to $\Theta_{1, 2\star 2, 3\star 0, 4\star 0}$, and $G$ is proper conflict-free $L$-colorable by Proposition~\ref{prop:theta graph}, a contradiction.
    Thus, $G'$ has no isolated vertices.
    Let $\phi$ be a proper conflict-free $L$-coloring of $G'$.
	For $v\in \{u_3,v_2\}$, let $c_v$ be a color in $\mathcal{U}_\phi(v,G')$.
	We extend $\phi$ to an $L$-coloring of $G$.
	
	We choose $\phi(u_0)\in L(u_0)\setminus \{\phi(u_3), c_{u_3}, \phi(v_1), \phi(v_2)\}$, choose $\phi(u_2)\in L(u_2)\setminus \{\phi(u_0),\phi(v_2),c_{v_2}\}$, and then choose $\phi(u_1)\in L(u_1)\setminus \{\phi(u_0),\phi(v_1),\phi(u_2)\}$.
	By the choice of colors, $\phi$ is a proper $L$-coloring such that $\mathcal{U}_\phi(v,G)\neq \emptyset$ for every $v\in V(G)\setminus \{u_0,v_1\}$.
	Each of $u_0$ and $v_1$ has three neighbors with at least two colors, and hence we have $\mathcal{U}_\phi(u_0,G)\neq \emptyset$ and $\mathcal{U}_\phi(v_1,G)\neq \emptyset$.
	Thus, $\phi$ is a proper conflict-free $L$-coloring of $G$, a contradiction.
\end{proof}

\begin{claim}\label{claim:45vertex 2thread}
	For $k\in \{4,5\}$, every $k$-vertex is incident with at most $k-2$ 2-threads of $G$.
\end{claim}

\begin{proof}
	For an integer $k\in \{4,5\}$, let $u_0$ be a $k$-vertex with neighbors $u_1,\dots ,u_k$.
	Suppose that for each $i\in \{1,\dots ,k-1\}$, $u_i$ is contained in a 2-thread $u_iv_i$.
    We fix such a vertex $u_0$ so that the degree $k$ is as small as possible.
	Let $w_i$ be the neighbor of $v_i$ other than $u_i$.
	Let $G'=G-\{u_0\}\cup \bigcup_{i=1}^{k-1}\{u_i,v_i\}$.
    By Proposition~\ref{prop:c5} and the minimality of $G$, $G'$ admits a proper conflict-free $L$-coloring $\phi$.
    By Claim~\ref{claim:pendant cycle}, $\{u_i,v_i\}\cap \{u_j,v_j\}=\emptyset$ for any distinct $i$ and $j$.
    If $G'$ has an isolated vertex $v$, then it must be that $N_G(v)\subseteq \{u_0,v_1,\dots ,v_{k-1}\}$.
    Furthermore, by Claims~\ref{claim:3plus neighbor} and the minimality of $k$, it follows that $N_G(v)=\{u_0,v_1,\dots ,v_{k-1}\}$.
    This implies that $G$ is isomorphic to $\Theta_{1,2\star 0, 3\star (k-1), 4\star 0}$, and hence $G$ is proper conflict-free $L$-colorable by Proposition~\ref{prop:theta graph}, a contradiction.
    
    Thus $G'$ has no isolated vertices.
	For each $v\in \{u_k, v_1,\dots ,v_{k-1}\}$, let $c_v$ be a color in $\mathcal{U}_\phi(v,G')$.
	We define a list assignment $L^*$ of $V(G)\setminus V(G')$ as follows:
	\begin{itemize}
		\item Let $L^*(u_0)=L(u_0)\setminus \{\phi(u_k),c_{u_k}\}$.
		\item For $i\in \{1,\dots ,k-1\}$, let $L^*(u_i)=L(u_i)\setminus \{\phi(w_i)\}$ and let $L^*(v_i)=L(v_i)\setminus\{\phi(w_i), c_{w_i}\}$.
	\end{itemize}
	By deleting extra colors, we may assume that $|L^*(u_0)|=k$, $|L^*(u_i)|=3$ and $|L^*(v_i)|=2$ for each $i\in \{1,\ldots, k-1\}$.
	
	For each $i \in \{1,\dots ,k-2\}$, we let $\phi(u_i)\in L^*(u_i)\setminus L^*(v_i)$.
	Note that $u_0$ has $k-1\in \{3,4\}$ colored neighbor at this time.
	For the remaining vertices, we consider the following two cases.
	
	\medskip
	\noindent
	\textit{Case 1.} $|\{\phi(u_1),\ldots, \phi(u_{k-2}), \phi(u_k)\}|\geq 3$.
	
	We first choose $\phi(u_0)\in L^*(u_0)\setminus \{\phi(u_1),\dots ,\phi(u_{k-2})\}$, choose $\phi(v_{k-1})\in L^*(v_{k-1})\setminus \{\phi(u_0)\}$ and choose $\phi(u_{k-1})\in L^*(u_{k-1})\setminus \{\phi(u_0),\phi(v_{k-1})\}$.
	Then, for each $i\in \{1,\dots ,k-2\}$, let $\phi(v_i)\in L^*(v_i)\setminus \{\phi(u_0)\}$.
	Obviously, $\phi$ is a proper $L$-coloring of $G$.
	Furthermore, since $d_G(u_0)\in \{4,5\}$ and at least three colors appear in the neighbors, we have $\mathcal{U}_\phi(u_0,G)\neq \emptyset$, and hence $\phi$ is a proper conflict-free $L$-coloring of $G$, a contradiction.
	
	\medskip
	\noindent
	\textit{Case 2.} $|\{\phi(u_1),\dots , \phi(u_{k-2}), \phi(u_k)\}|\leq 2$.
	
	We first choose $\phi(u_{k-1})\in L^*(u_{k-1})\setminus \{\phi(u_1), \ldots, \phi(u_{k-2}), \phi(u_k)\}$ and choose $\phi(v_{k-1})\in L^*(v_{k-1})\setminus\{\phi(u_{k-1})\}$.
	Then, since $\phi(u_k)\notin L^*(u_0)$, it follows that
	$|L^*(u_0)\setminus \{\phi(u_1),\dots , \phi(u_{k-2}), \phi(u_k)\}|\geq |L^*(u_0)|-1\geq k-1\geq 3$, and hence we can choose $\phi(u_0)\in L^*(u_0)\setminus (\{\phi(u_1),\dots , \phi(u_{k-2}), \phi(u_k)\}\cup \{\phi(u_{k-1}),v_{k-1}\})$.
	Finally, we let $\phi(v_i)\in L^*(v_i)\setminus\{\phi(u_0)\}$ for each $i\in \{1,\dots ,k-2\}$.
	By the choice of colors, it is easy to verify that $\phi$ is a proper conflict-free $L$-coloring of $G$, a contradiction.
\end{proof}

\begin{claim}\label{claim:4vertex 2thread}
	Let $u_0$ be a 4-vertex with $N_G(u_0)=\{u_1, u_2, u_3, u_4\}$ such that $d_G(u_1)=d_G(u_2)=d_G(u_3)=2$.
	If both $u_1$ and $u_2$ are contained in 2-threads of $G$, then the neighbor of $u_3$ other than $u_0$ is a $4^+$-vertex.
\end{claim}

\begin{proof}
	For each $i\in \{1,2\}$, let $u_iv_i$ be the 2-thread of $G$ containing $v_i$ and let $w_i$ be the neighbor of $v_i$ other than $u_i$.
	Let $v_3$ be the neighbor of $u_3$ other than $u_0$.
	Suppose that $d_G(v_3)\leq 3$.
	By Claim~\ref{claim:45vertex 2thread}, we know that $d_G(v_3)=3$.
	
	Let $G'=G-\{u_0,u_1,u_2,u_3,v_1,v_2\}$.
    If $G'$ has an isolated vertex $v$, then it must be that $N_G(v)\subseteq \{u_0,u_3,v_1,v_2\}$.
    However, as $d_G(v_3)=3$, we have that $d_G(v)\leq 3$, a contradiction by Claim~\ref{claim:3vertex 2thread} (when $d_G(v)=3$) or Claim~\ref{claim:pendant cycle} (when $d_G(v)=2$).

    Thus, $G'$ has no isolated vertices.
    By Proposition~\ref{prop:c5} and the minimality of $G$, $G'$ admits a proper conflict-free $L$-coloring $\phi$.
    Note that two neighbors of $v_3$ in $G'$ have distinct colors.
	For each $v\in \{u_4,w_1,w_2\}$, let $c_v$ be a color in $\mathcal{U}_\phi(v,G')$.
	We define a list assignment $L^*$ as follows:
	\begin{itemize}
		\item Let $L^*(u_0)=L(u_0)\setminus \{\phi(u_4), \phi(v_3), c_{u_4}\}$,
		\item For $i\in \{1,2\}$, let $L^*(u_i)=L(u_i)\setminus \{\phi(w_i)\}$ and let $L^*(v_i)=L(v_i)\setminus\{\phi(w_i), c_{w_i}\}$.
		\item Let $L^*(u_3)=L(u_3)\setminus \{\phi(v_3)\}$.
	\end{itemize}
	By deleting extra colors, we may assume that $|L^*(u_0)|=|L^*(u_1)|=|L^*(u_2)|=|L^*(u_3)|=3$ and $|L^*(v_1)|=|L^*(v_2)|=2$.
	
	We extend $\phi$ to an $L$-coloring of $G$ by the following manner.
	We first let $\phi(u_1)\in L^*(u_1)\setminus L^*(v_1)$ and $\phi(u_2)\in L^*(u_2)\setminus L^*(v_2)$, and let $\phi(u_0)\in L^*(u_0)\setminus \{\phi(u_1),\phi(u_2)\}$.
	Since three neighbors $\{u_1,u_2,u_4\}$ are colored, either all vertices of $\{u_1,u_2,u_4\}$ are colored in the same color $\alpha$, or there exists a color $\beta$ that appears exactly once in $\{u_1,u_2,u_4\}$.
	In the former case, we let $\phi(u_3)\in L^*(u_3)\setminus \{\phi(u_0), \alpha\}$.
	In the latter case, we let $\phi(u_3)\in L^*(u_3)\setminus \{\phi(u_0), \beta\}$.
	Then we have $\mathcal{U}_\phi(u_0,G)\neq \emptyset$.
	Finally, we let $\phi(v_i)\in L^*(v_i)\setminus \{\phi(u_0)\}$ for $i\in \{1,2\}$.
    As $v_3$ is a 3-vertex with at least two distinct colors in its neighbors, we have $\mathcal{U}_{\phi}(v_3,G)\neq \emptyset$.
    It is easy to verify that $\phi$ is a proper conflict-free $L$-coloring of $G$, a contradiction.
\end{proof}

\subsection{Discharging}

We set an initial charge $c:V(G)\to \mathbb{R}$ as $c(v)=7d_G(v)-18$ for each vertex $v\in V(G)$.
We move the charges of vertices according to the following discharging rules.

\begin{enumerate}[label=(R\arabic*)]
	\item\label{rule:3} Let $u$ be a 3-vertex of $G$.
	\begin{itemize}
		\item If $u$ has two $2$-neighbors, then $u$ sends charge $\frac{3}{2}$ to each of them.
		\item If $u$ has one $2$-neighbor, then $u$ sends charge $2$ to it.
	\end{itemize}
	\item\label{rule:4plus} Let $u$ be a $4^+$-vertex of $G$.
	\begin{itemize}
		\item $u$ sends charge $2$ to each vertex of an incident maximal 3-thread.
		\item $u$ sends charge $2$ to each vertex of an incident maximal 2-thread.
		\item $u$ sends charge $\frac{5}{2}$ to an adjacent 2-vertex which has a 3-neighbor other than $u$.
		\item $u$ sends charge $2$ to an adjacent 2-vertex which has a $4^+$-neighbor other than $u$.
	\end{itemize}
\end{enumerate}

Let $c^*(v)$ be the charge of each vertex $v\in V(G)$ after the discharge.
We shall show that $c^*(v)\geq 0$ for each $v\in V(G)$.

\begin{claim}\label{claim:charge 2vertex}
	For every 2-vertex $v$, $c^*(v)\geq 0$.
\end{claim}

\begin{proof}
    Let $P$ be a maximal thread of $G$ that contains $v$, which is incident with two $3^+$-vertices $x$ and $y$. By Claim~\ref{claim:long thread}, $k:=|V(P)|\leq 3$.
	
	If $k\in \{2,3\}$, then by Claim~\ref{claim:3vertex 2thread}, we have $d_G(x) \geq 4$ and $d_G(y) \geq 4$.
	Thus, by \ref{rule:4plus}, each of $x$ and $y$ sends charge $2$ to $v$, and hence $c^*(v)=c(v)+2\times 2=14-18+4=0$. 
    Suppose that $k=1$, that is, $x$ and $y$ are two neighbors of $v$. 
	If $d_G(x)= d_G(y) = 3$, then Claim~\ref{claim:3vertex bad neighbor} implies that both $x$ and $y$ are 3-vertices with exactly one 2-neighbor.
	Thus, by \ref{rule:3}, each of $x$ and $y$ sends charge $2$ to $v$, and hence $c^*(v)=c(v)+2\times 2=14-18+4=0$. Thus, without loss of generality, we may assume that $d_G(y) \geq 4$. Then by \ref{rule:3} and \ref{rule:4plus}, $x$ sends charge at least $\frac{3}{2}$ to $v$ and $y$ sends charge $\frac{5}{2}$ to $v$, and hence $c^*(v)\geq c(v)+\frac{3}{2}+\frac{5}{2}=14-18+4=0$.
\end{proof}

\begin{claim}\label{claim:charge 3vertex}
	For every 3-vertex $v$, $c^*(v)\geq 0$.
\end{claim}

\begin{proof}
	By Claim~\ref{claim:3plus neighbor}, $v$ has at most two 2-neighbors.
	If $v$ has two 2-neighbors, then $c^*(v)=c(v)-2\times \frac{3}{2}=21-18-3=0$ by \ref{rule:3}.
	Otherwise, $c^*(v)\geq c(v)-2=21-18-2>0$ by \ref{rule:3}.
\end{proof}

\begin{claim}\label{claim:charge 4vertex}
	For every $4^+$-vertex $v$, $c^*(v)\geq 0$.
\end{claim}

\begin{proof}
	Let $v$ be a $k$-vertex for some $k\geq 4$.
	We define $n_0$, $n_1$, $n_2$ and $n_3$ as follows:
	\begin{align*}
		 n_0 & :=|\{u\in N_G(v)\mid u\text{ is a 2-vertex with a }4^+\text{-neighbor other than }v\}|, \\
		 n_1 & :=|\{u\in N_G(v)\mid u\text{ is a 2-vertex with a 3-neighbor other than }v\}|, \\
		 n_2 & :=|\{u\in N_G(v)\mid u\text{ is contained in a maximal 2-thread of }G\}|, \text{\;and} \\
		 n_3 & :=|\{u\in N_G(v)\mid u\text{ is contained in a maximal 3-thread of }G\}|.
	\end{align*}
	By Claim~\ref{claim:3plus neighbor}, $v$ has at most $k-1$ 2-neighbors, and hence 
	$n_0+n_1+n_2+n_3\leq k-1$.
	According to the discharging rule \ref{rule:4plus}, we have $c^*(v)=c(v)-2n_0-\frac{5}{2}n_1-4n_2-6n_3$.
	
	Suppose that $k=4$. Claims~\ref{claim:3thread 1} and \ref{claim:45vertex 2thread} respectively implies that $n_3\leq 1$ and $n_2+n_3\leq 2$.
    If $n_0+n_1+n_2+n_3=3$, then Claim~\ref{claim:3thread 2} implies that $n_3=0$.
    If $n_0+n_1+n_2+n_3=3$ and $n_2=2$, Claim~\ref{claim:4vertex 2thread} implies that $n_1=0$.
    By these arguments, only possibilities around $v$ are that
	\begin{itemize} 
		\item $n_0+n_1+n_2+n_3\leq 2$ and $n_3\leq 1$,
		\item $n_0+n_1+n_2+n_3=3$, $n_2\leq 1$, and $n_3=0$, or
		\item $n_0=1$, $n_1=0$, $n_2=2$, and $n_3=0$.
	\end{itemize}
	For each case, we have $c^*(v)\geq c(v)-4\times 1-6\times 1=28-18-10=0$, $c^*(v)\geq c(v)-\frac{5}{2}\times 2-4\times 1=28-18-9 > 0$, and $c^*(v)=c(v)-2\times 1-4\times 2=28-18-10=0$, respectively.

	Suppose that $k=5$. Claims~\ref{claim:3thread 1} and \ref{claim:45vertex 2thread} respectively implies that $n_3\leq 2$ and $n_2+n_3\leq 3$. By Claim~\ref{claim:3thread 2}, if $n_0+n_1+n_2+n_3=4$, then $n_3 \leq 1$. Thus, only possibilities around $v$ are that 
	\begin{itemize}
		\item $n_0+n_1+n_2+n_3\leq 3$ and $n_3\leq 2$, or
		\item $n_0+n_1+n_2+n_3=4$, $n_2+n_3\leq 3$, and $n_3\leq 1$.
	\end{itemize}
	For the former case, we have $c^*(v)\geq c(v)-4\times 1-6\times 2=35-18-16>0$. For the latter case, we have $c^*(v)\geq c(v)-\frac{5}{2}\times 1-4\times 2-6\times 1=35-18-\frac{33}{2}>0$.
	
	Finally, suppose that $k\geq 6$.
	By Claims~\ref{claim:3thread 1} and \ref{claim:3thread 2}, only possibilities around $v$ are that 
	\begin{itemize}
		\item $n_0+n_1+n_2+n_3\leq k-2$ and $n_3\leq k-3$, or
		\item $n_0+n_1+n_2+n_3=k-1$ and $n_3\leq k-4$.
	\end{itemize}
	For the former case, we have $c^*(v)\geq c(v)-4\times 1-6\times (k-3)=7k-18-4-6k+18=k-4>0$.
    For the latter case, we have $c^*(v)\geq c(v)-4\times 3-6\times (k-4)=7k-18-12-6k+24=k-6\geq 0$.
\end{proof}

Combining Claims~\ref{claim:charge 2vertex}, \ref{claim:charge 3vertex}, and \ref{claim:charge 4vertex}, we have that
\[0\leq \sum_{v\in V(G)}c^*(v)=\sum_{v\in V(G)}c(v)=\sum_{v\in V(G)}(7d_G(v)-18)=14|E(G)|-18|V(G)|.\]
However, this implies that ${\rm mad}(G)\geq \frac{2|E(G)|}{|V(G)|}\geq \frac{18}{7}$, a contradiction.
This completes the proof of Theorem~\ref{thm:mad 18 over 7}.

\section{Concluding remark}
In this paper, we investigate proper conflict-free $({\rm degree}+k)$-choosability of graphs with a bounded maximum average degree.
Since we have only few examples that are `not' proper conflict-free $({\rm degree}+k)$-choosable for small $k$, we wonder if the following common strengthening of Conjectures~\ref{conj:degree plus 1} and \ref{conj:degree plus 2} still may hold.

\begin{conjecture}
    Let $G$ be a connected graph distinct from $C_5$ and let $L$ be a list assignment of $G$ such that 
    \begin{equation*}
        |L(v)|=\begin{cases}
            4\quad &\text{if}\quad d_G(v)=2\quad \text{and}\\
            d_G(v)+1\quad &\text{if}\quad d_G(v)\neq 2.
        \end{cases}
    \end{equation*}
    Then $G$ is proper conflict-free $L$-colorable.
\end{conjecture}
 
\section*{Acknowledgement}
 
M. Kashima has been supported by JSPS KAKENHI JP25KJ2077.
R. \v{S}krekovski has been partially supported by the Slovenian Research Agency ARIS program P1-0383 and ARIS project J1-3002 and J1-4351. R. Xu has been supported by National Science Foundation for Young Scientists of China, Grant Number: 12401472.

\end{document}